\documentclass{amsart}
\usepackage{graphicx}
\usepackage{amssymb}
\usepackage[font={footnotesize,it}]{caption}
\usepackage{comment}
\usepackage{url}
\newtheorem{theorem}{Theorem}

\newtheorem{lemma}[theorem]{Lemma}

\newtheorem{corollary}[theorem]{Corollary}
\numberwithin{theorem}{section}
\theoremstyle{remark}
\newtheorem*{remark}{Remark}

\DeclareMathOperator{\sgn}{sgn}

\begin{document}
\title{An alternative to Riemann-Siegel type formulas}
\author{Ghaith A.\ Hiary}
\thanks{Preparation of this material is partially supported by
the National Science Foundation under agreement No.\ 
 DMS-0932078 (while at MSRI) and DMS-1406190, 
and by the Leverhulme Trust (while at the University of Bristol).}
\address{Department of Mathematics, The Ohio State University, 231 West 18th Ave,
Columbus, OH 43210}
\begin{abstract}
Simple unsmoothed formulas to compute the Riemann zeta function, 
and Dirichlet $L$-functions to a power-full
modulus, are derived by elementary means (Taylor expansions and the geometric series).
The formulas enable square-root of the analytic conductor complexity,
up to logarithmic loss, 
and have an explicit remainder term that is easy to control.
The formula for zeta yields
a convexity bound of the same strength as that from the Riemann-Siegel formula, 
up to a constant factor. 
Practical parameter choices are discussed.
\end{abstract}
\maketitle

\section{Introduction}
The Riemann zeta function is defined for $s=\sigma+it$ by 
$\zeta(s) = \sum_{n=1}^{\infty} n^{-s}$, $\sigma>1$.
It can be analytically continued everywhere 
except for a simple pole at $s=1$.
%with residue $1$.
The zeta function satisfies the functional equation 
$\zeta(s) = \chi(s) \zeta(1-s)$ where $\chi(s) := \pi^{s-1/2}
\Gamma((1-s)/2)/\Gamma(s/2)$.
One is usually interested in numerically 
evaluating $\zeta(\sigma+it)$ on the critical line $\sigma = 1/2$ 
(e.g.\ to verify the Riemann hypothesis). 
However, one cannot use the Dirichlet series $\sum_{n=1}^{\infty} n^{-s}$
to numerically evaluate zeta
when $\sigma<1$ because the series diverges.
Rather, one can use partial summation and integration by parts
to analytically continue the series to $\sigma>0$, obtaining
\begin{equation} \label{zeta trunc}
\zeta(s) = \sum_{1\le n<M} \frac{1}{n^s} +\frac{M^{-s}}{2}+ \frac{M^{1-s}}{s-1}
+\mathcal{R}_M(s), \qquad |\mathcal{R}_M(s)| \le \frac{\mathfrak{q}(s)}{\sigma
M^{\sigma}},
\end{equation}
where $\mathfrak{q}(s):=|s|+3$ is the analytic conductor of zeta; 
see \cite{iwaniec-kowalski}.
The analytic conductor terminology  
was introduced by Iwaniec and Sarnak; see \cite{iwaniec-sarnak} for example.
This terminology will be useful when we generalize our formulas to Dirichlet
$L$-functions, and it ensures that $\log\mathfrak{q}(s)>0$.
We remark, though, that the precise definition of the analytic conductor does not affect
the asymptotic content of the results,
since $\mathfrak{q}(s)$ needs only be of a comparable size to $|s|$.

Formula \eqref{zeta trunc} can be viewed as consisting of a main sum
$\sum_{n<M} n^{-s}$, an extra term $M^{-s}/2+M^{1-s}/(s-1)$, and 
a remainder $\mathcal{R}_M(s)$. 
The main sum accounts for the bulk of the computational effort, 
the extra term can be computed easily,
and the remainder can be controlled by choosing $M$ accordingly.
For example, one can ensure that $|\mathcal{R}_M(s)|<\epsilon$ on taking
$M > (\mathfrak{q}(s)/(\sigma \epsilon))^{1/\sigma}$.
So when $\sigma = 1/2$, the main sum consists of 
$\gg \mathfrak{q}(s)^2$ terms, even if $\epsilon = 1$ say.
Using a more careful analysis, however, 
one can show that $\mathcal{R}_M(s) \ll M^{-\sigma}$ 
if $M \gg \mathfrak{q}(s)$. 
Alternatively, one can use the Euler-Maclaurin summation (see \textsection{\ref{prev methods}})
which allows for far more accuracy. In either case, though, 
the resulting main sum is of length $\gtrsim \mathfrak{q}(s)$.
%which involves the Poisson summation formula, 
%can show that $\mathcal{R}_M(s) \ll M^{-\sigma}$ for $\sigma\ge \sigma_0 > 0$ 
%and $M> M_0 = C |t|/(2\pi)$ when $C$ is a given constant greater than $1$; see
%\cite[Chapter IV]{titchmarsh-book}. According to this estimate, 
%one can ensure that $\mathcal{R}_M(s) \ll
%\mathfrak{q}(s)^{-1}$, say, by taking $M\gg \max\{M_0,\mathfrak{q}(s)\}$. 
%while one can use formula \eqref{zeta trunc} 
%to obtain a numerical approximation of $\zeta(1/2+it)$, 
%quite long because of the slow 
% convergence of $\sum_{n<M}n^{-s} +M^{1-s}/(s-1)$.
So these formulas are
rather impractical for numerical computations on single processor
when $t \gtrsim 10^{10}$, say, especially if high precision is sought. 
This is unfortunate since they are simple 
to derive and analyze, and have explicit error bounds.
So, instead, one typically uses the Riemann-Siegel asymptotic formula 
which has a much shorter main
sum of length $\lfloor \sqrt{t/(2\pi)}\rfloor$ (see \textsection{\ref{prev
methods}}). 
The Riemann-Siegel formula was discovered around $1932$ in
Riemann's unpublished papers by C.L. Siegel.
Some of its history is narrated in \cite[Chapter 7]{edwards-zeta-book}. 
In lieu of the Riemann-Siegel formula, one can use 
the efficient smoothed formulas in \cite{rubinstein-computational-methods}.

We propose a new method for computing zeta based on slowly converging
Dirichlet series such as \eqref{zeta trunc}. Then we generalize our method to
Dirichlet $L$-functions to a power-full modulus.
Interestingly, our results can be derived without knowing about the functional equation of the
associated $L$-function, nor using analysis of similar strength,
such as the Poisson summation.  
To state the results, we introduce some notation. Let
\begin{equation}
\begin{split}
f_s(z) := \frac{e^{sz}}{(1+z)^s},\,\,\,  f_s(0)=1,\quad
g_K(z) := \sum_{k=0}^{K-1} e^{kz} =\frac{e^{Kz}-1}{e^z-1},\,\, z\not\in 2\pi
i \mathbb{Z},
\end{split}
\end{equation}
where  
$f_s^{(j)}(z)$ and $g_K^{(j)}(z)$ denote the $j$-th derivative
in $z$. 
We choose integers $u_0\ge 1$, $v_0\ge u_0$, and $M\ge v_0$,
and construct sequences $K_r= \lceil v_r/u_0\rceil$ 
and $v_{r+1}= v_r + K_r$ for $0\le r< R$, where 
$R:=R(v_0,u_0,M)$ is the largest integer such that $v_R < M$.
We define $K_R :=\min\{\lceil v_R/u_0\rceil,M-v_R\}$,
so that $v_{R+1} = M$. Then we divide the main sum in \eqref{zeta trunc} 
into an initial sum of length $v_0$, followed by
 $R+1$ consecutive blocks where
the $r$-th block starts at $v_r$ and has length $K_r$.
The sequences $K_r$ and $v_r$ are so defined in order to implement
 a more efficient version of dyadic subdivision of the main sum.
 There will be substantial flexibility in choosing them
%$v_r$ and $K_r$ 
(need only $K_r-1\le v_r/u_0$, $u_0\ge \sqrt{\mathfrak{q}(s)}$), 
but we do not exploit this here. 
We plan to approximate the $r$-th block $\sum_{v_r\le n<v_r+K_r} n^{-s}$ 
by $v_r^{-s}B_r(s,m)$ where
\begin{equation}
B_r(s,m):=\sum_{j=0}^m \frac{f_s^{(j)}(0)}{j!},\frac{g_{K_r}^{(j)}(-s/v_r)}{v_r^j}, 
\end{equation}
which is  a linear combination of a geometric sum and its derivatives. 
Also, we let
\begin{equation}
\mathcal{B}_M(s,u_0,v_0):=\sum_{r=0}^R v_r^{-\sigma}
\min\{g_{K_r}(-\sigma/v_r),|\csc(t/(2v_r))|\},
\end{equation}
\begin{equation} \label{eps def}
\epsilon_m(s,u):=\left\{\begin{array}{ll}
\displaystyle 
\frac{3.5\, e^{0.78(m+1)}}{(m+1)^{(m+1)/2}}\frac{|s|^{(m+1)/2}}{u^{m+1}},\,\, &m\le |s|/4,\\
\\
\displaystyle \frac{2^m e^{0.194|s|}}{u^m},\,\, &m> |s|/4.
\end{array}\right.
\end{equation}
We prove the following theorem in \textsection{\ref{proofs}}. 
\begin{theorem}\label{zeta alg}
Given $s=\sigma+it$ with $\sigma>0$,  
let $u_0$ and $v_0$ be any integers satisfying
 $v_0\ge u_0\ge 2\max\{6,\sqrt{\mathfrak{q}(s)},\sigma\}$.
Then for any integers $M\ge v_0$ and  $m\ge 0$ we have 
\begin{displaymath}
\zeta(s) = \sum_{n=1}^{v_0-1} \frac{1}{n^s} 
+ \sum_{r=0}^R \frac{B_r(s,m)}{v_r^s}
+\frac{M^{-s}}{2}+ \frac{M^{1-s}}{s-1} 
+ \mathcal{T}_{M,m}(s,u_0,v_0) + \mathcal{R}_M(s),
\end{displaymath}
where $\displaystyle |\mathcal{T}_{M,m}(s,u_0,v_0)|  
\le \epsilon_m(s,u_0) \mathcal{B}_M(s,u_0,v_0)$. 
We have $R < 2u_0 \log(M/v_0)+1$.
\end{theorem}
We could have used the main sum from the Euler-Maclaurin formula, instead
of the main sum in \eqref{zeta trunc},
to derive Theorem~\ref{zeta alg}. This permits one to choose  $M$ smaller.
Indeed, replacing $\mathcal{R}_M(s)$ by the Euler-Maclaurin correction terms, 
one can restrict $M\ll \mathfrak{q}(s)$ while retaining high accuracy. 
In this case, Theorem~\ref{zeta alg}, applied with $m=0$, leads to a simple
proof of the bound $\zeta(1/2+it)\ll \mathfrak{q}(1/2+it)^{1/4}$; see
corollary~\ref{zeta bound} in \textsection{\ref{convexity-bounds}}. 
The truncation error $\mathcal{T}_{M,m}$ in Theorem~\ref{zeta alg} 
is  bounded by $\epsilon_m\mathcal{B}_M$, where, by lemma~\ref{Kr bound}, we have 
$\mathcal{B}_M(s,u_0,v_0)\le v_0^{-\sigma} + 
(M^{1-\sigma}-v_0^{1-\sigma})(1-\sigma)^{-1}$ if $\sigma \ne 1$,
and $\mathcal{B}_M(s,u_0,v_0) \le v_0^{-\sigma}+\log(M/v_0)$ if $\sigma = 1$.
This estimate is quite generous, however. It can be improved 
by computing $\mathcal{B}_M(s,u_0,v_0)$ directly, 
which should yield a bound like $u_0/(\sigma v_0^{\sigma})$.
The said computation can be done in about
$R$ steps, and so it is subsumed by the computational
effort for the main sum. 
In either case, the remainder term is clearly
easy to control when $u_0\ge\sqrt{\mathfrak{q}(s)}$,
 due to the rapid decay of $\epsilon_m(s,u_0)$ with $m$
(decays like $1/\lfloor (m+1)/2\rfloor!$).
%even though we do not use a smoothing function.
%We will typically choose $u_0$ and $v_0$ about
%$\sqrt{\mathfrak{q}(s)}$.
%So, $R\ll \sqrt{\mathfrak{q}(s)}\log (M/\epsilon)$
%and $m\ll \log (M/\epsilon)$, 
%where $\epsilon>0$ is the error tolerance.
%In particular, 

The main sum in Theorem~\ref{zeta alg} has $v_0 + (m+1)(R+1)$ terms,
where each term is, basically, a geometric sum.
To ensure that $|\mathcal{T}_{M,m}(s)| + |\mathcal{R}_M(s)| <\epsilon$ 
for $\sigma=1/2$, 
it suffices to take $M \ll (\mathfrak{q}(s)/\epsilon)^2$ and 
$m \ll \log (\mathfrak{q}(s)/\epsilon)$. Since $R\le
2u_0\log(M/v_0)+1$, this is of length 
$\ll v_0+u_0\log^2 (\mathfrak{q}(s)/\epsilon)$ terms.
Choosing $u_0=v_0 = 2\lceil \sqrt{\mathfrak{q}(s)}\rceil$, which is
a typical choice, the main sum thus consists of
$\ll \sqrt{\mathfrak{q}(s)} \log^2 (\mathfrak{q}(s)/\epsilon)$ terms.
We show how to compute these terms (geometric sums) 
efficiently in \textsection{\ref{terms comp}},
using $\ll \log(\mathfrak{q}(s)/\epsilon)$ precision.
So, put together, the complexity of the formula in Theorem~\ref{zeta alg} depends 
only logarithmically on $M$ and the error tolerance $\epsilon$. 
The formula enables square-root of the analytic conductor complexity, up to
logarithmic loss, without using the
functional equation, or the approximate functional equation.
Also, the usual factor $\chi(s)$ does not appear, and 
the conditions on $v_0$ and $u_0$ imply that $v_0u_0 \gg \mathfrak{q}(s)$.
Nevertheless, the idea behind the theorem 
is fairly simple. Writing $n^{-s}=e^{-s\log n}$, we have $\sum_{v\le n<v+K}
n^{-s}= v^{-s}\sum_{0\le k<K} e^{-s\log(1+k/v)}$. So if $K/v \ll
1/\sqrt{\mathfrak{q}(s)}$, as we will have, then 
$s\log(1+k/v) = sk/v + O(1)$. In particular, using Taylor expansions,  
we can approximate $\sum_{0\le k<K} e^{-s\log(1+k/v)}$ 
by a linear combination of the  
geometric sum $g_K(-s/v)$ and several of its
derivatives. These geometric sums are easy to compute, which is
the reason for the savings. 

One can shorten the length of the main sum in Theorem~\ref{zeta alg}
to be roughly $\mathfrak{q}(s)^{1/3}$.
But then instead of obtaining linear
exponential sums, one obtains quadratic exponential sums.
The length can be further shortened, leading to cubic and higher
degree exponential sums.
In view of this, Theorem~\ref{zeta alg} belongs to
the family of methods for computing zeta that were derived in
\cite{hiary-fast-methods}.
And like these methods (see \cite{hiary-char-sums}), 
Theorem~\ref{zeta alg} can be generalized to Dirichlet 
$L$-functions $L(s,\chi)$, $\chi\bmod{q}$, when $q$ is power-full. 
To this end, define the analytic conductor
for $L(s,\chi)$ by $\mathfrak{q}(s,\chi) := q (|s|+3)$.
If $\chi\bmod{q}$ is non-principal, then we have 
the trivial bound $|\sum_n \chi(n)|< q$. 
Combined with partial summation we obtain, for $\sigma>0$,
that\footnote{
To estimate $\mathcal{R}_M(s,\chi)$, we used the following partial summation formula (see
\cite{rubinstein-computational-methods}): Let $f:\mathbb{Z}^+\to\mathbb{C}$
and $g:\mathbb{R}\to\mathbb{C}$ such that $g'$ exists on $[1,x]$. Then
for $y\in [1,x]$ we have
\begin{equation*}
\sum_{y< n\le x} f(n)g(n) = 
\left(\sum_{y< n \le x} f(n)\right)g(x) +\left(\sum_{1\le n \le y}
f(n)\right)(g(x)- g(y)) - 
\int_y^x \left(\sum_{1\le n \le \tau} f(n)\right) g'(\tau) \,d\tau.
\end{equation*}}
\begin{equation}\label{dirichlet trunc}
L(s,\chi) = \sum_{1\le n<M} \frac{\chi(n)}{n^s} + \mathcal{R}_M(s,\chi),
\qquad |\mathcal{R}_M(s,\chi)| \le 
\frac{2\mathfrak{q}(s,\chi)}{\sigma M^{\sigma}}.
\end{equation}
We will only consider the case $q=p^a$ for $p$ prime. 
As in Theorem~\ref{zeta alg}, we divide the main sum in \eqref{dirichlet trunc}
into an initial sum of length $v_0$, followed by
$R+1$ consecutive blocks, where the $r$-th block starts at $v_r$ and has length
$K_r$. Let $g_K(z,\chi,v) := \sum_{0\le k<K} \chi(v+k) e^{kz}$.
Then, in analogy with zeta, we 
 approximate the $r$-th block $\sum_{v_r\le n<v_r+K_r} \chi(n) n^{-s}$ 
by $v_r^{-s}B_r(s,\chi,m)$ where
\begin{equation}
B_r(s,\chi,m):=\sum_{j=0}^m \frac{f_s^{(j)}(0)}{j!}
\frac{g_{K_r}^{(j)}(-s/v_r,\chi,v_r)}{v_r^j}, 
\end{equation}
and $g_K^{(j)}(z,\chi,v)$ denotes the $j$-th derivative in $z$. 
The analogue of $\mathcal{B}_M$ from Theorem~\ref{zeta alg} 
is going to be more complicated to define. 
To this end, let $b:=\lceil a/2\rceil$ and,
for $0\le d<p^b$, let
 $H_{r,d}:= \lceil (K_r-d)/p^b\rceil$ and
$w_{r,d}:=2\pi \overline{v_r+d}L/p^{a-b}$ 
where $L$ is as in lemma~\ref{gKr lemma} 
and $(\overline{v_r+d})(v_r+d)\equiv 1\bmod{p^a}$ if $\gcd(v_r+d,p)=1$.
Then let
\begin{equation}
\begin{split}
\mathcal{B}_M(s,\chi,u_0,v_0):=\sum_{r=0}^R 
\sum_{d=0}^{p^b-1}
\delta_{\gcd(v_r+d,p)=1}
&\min\{e^{-\sigma d/v_r}g_{H_{r,d}}(-p^b\sigma/v_r),\\
&|\csc(w_{r,d}/2-p^bt/(2v_r))|\}
v_r^{-\sigma}.
\end{split}
\end{equation}
In \textsection{\ref{proofs}}, we prove the following.
\begin{theorem}\label{dirichlet alg}
Given $s=\sigma+it$ with $\sigma>0$,
a non-principal Dirichlet character $\chi\bmod{p^a}$ with
$p$ a prime, let $b=\lceil a/2\rceil$, and let 
$u_0$ and $v_0$ be any integers satisfying
$v_0\ge u_0\ge 2\max\{6,\sqrt{\mathfrak{q}(s)},\sigma\}$.
Then for any integers $M\ge v_0$ and  $m\ge 0$ we have 
\begin{displaymath}\label{dirichlet formula}
\begin{split}
L(s,\chi) =& \sum_{n=1}^{v_0-1} \frac{\chi(n)}{n^s} 
+ \sum_{r=0}^{R} \frac{B_r(s,\chi,m)}{v_r^s}
+ \mathcal{T}_{M,m}(s,\chi) + \mathcal{R}_M(s,\chi), 
\end{split}
\end{displaymath}
where $\displaystyle|\mathcal{T}_{M,m}(s,\chi)|  
\le \epsilon_m(s,u_0) \mathcal{B}_M(s,\chi,u_0,v_0)$. 
We have $R < 2u_0 \log(M/v_0)+1$. 
\end{theorem}
We  use the Postnikov character formula
in \textsection{\ref{proofs}}
to show that $g_K(z,\chi,v)$ can be written as a sum of $p^b$ geometric sums.
\begin{lemma} \label{gKr lemma}
Given a Dirichlet character $\chi\bmod{p^a}$ with 
$p$ a prime, let $b=\lceil a/2\rceil$, and
$H_d:= \lceil (K-d)/p^b\rceil$. Then
$g_K(z,\chi,v) = \sum_{d=0}^{p^b-1} \chi(v+d) e^{zd} g_{H_d}(p^bz+iw_d)$,
where $w_d:=2\pi \overline{v+d}L/p^{a-b}$ 
if $(v+d,p)=1$, with $(\overline{v+d})(v+d)\equiv 1\bmod{p^a}$, 
otherwise $w_d:=0$. 
Here, $L\in [0,p^{a-b})$ is the integer
determined by the equation $\chi(1+p^b)=e^{2\pi i L/p^{a-b}}$.
\end{lemma}
The main sum in Theorem~\ref{dirichlet alg} 
has $\le v_0 + (m+1)(R+1)p^b$ terms, where the
extra $p^b$  is from the formula for $g_K(z,\chi,v)$
in lemma~\ref{gKr lemma}.
One can easily deduce from the proof of lemma~\ref{Kr bound} that 
$\mathcal{B}_M(s,\chi,u_0,v_0) \le v_0^{-\sigma}+ (M^{1-\sigma}-v_0^{1-\sigma})(1-\sigma)^{-1}$
if $\sigma\ne 1$, and $\mathcal{B}_M(s,\chi,u_0,v_0)\le v_0^{-\sigma} +\log(M/v_0)$
if $\sigma = 1$. This bound is generous, of course, and can be improved
by computing $\mathcal{B}_M(s,\chi,u_0,v_0)$ directly, as was pointed out
earlier for zeta. In any case,
we can ensure that $|\mathcal{T}_{M,m}(s,\chi)| + |\mathcal{R}_M(s,\chi)| <\epsilon$
for $\sigma=1/2$, by taking
 $M \ll (\mathfrak{q}(s,\chi)/\epsilon)^2$ and
$m \ll \log (\mathfrak{q}(s,\chi)/\epsilon)$. 
So, choosing $u_0 = 2\lceil \sqrt{\mathfrak{q}(s)}\rceil$ 
and $v_0 = p^b u_0$, we see that the main
sum on the critical line can be made of length
$\ll p^b\sqrt{\mathfrak{q}(s)} \log^2 (\mathfrak{q}(s,\chi)/\epsilon)$ terms.
If $a$ is an even integer, or a large integer, then $p^b\approx \sqrt{q}$, and
so the length of the main sum
is about $\sqrt{\mathfrak{q}(s,\chi)}\log^2 (\mathfrak{q}(s,\chi)/\epsilon)$.
We remark that one can apply the Euler-Maclaurin formula 
along arithmetic progressions to the main sum in \eqref{dirichlet trunc}
(for each residue class of $p^a$).
This way, one can restrict $M\ll \mathfrak{q}(s,\chi)$, replacing
$\mathcal{R}_M(s,\chi)$ by
the correction terms resulting from the Euler-Maclaurin formula.
These correction terms 
will involve sums over the residue classes of $p^a$. 
But it will not be too hard to see that these sums can be tackled
using the same methods presented here.

\begin{remark}
If $\sigma > 0$, then one has the exact expression
\begin{equation}
L(s,\chi) = \sum_{n=1}^{v_0-1} \frac{\chi(n)}{n^s} 
+ \sum_{r=0}^{\infty}\frac{1}{v_r^s} 
\sum_{j=0}^{\infty} \frac{f_s^{(j)}(0)}{j!}
\frac{g_{K_r}^{(j)}(-s/v_r,\chi,v_r)}{v_r^j}.
\end{equation}
The order of the double sum can be switched if $\sigma>1$.
\end{remark}

\section{Previous methods and motivation} \label{prev methods}
In the case of the Riemann zeta function, 
one can use the Euler-Maclaurin summation
to obtain a main sum of length about $\mathfrak{q}(s)$. 
One notes that
$n^{-s}$ changes slowly with $n$ when $n\gg \mathfrak{q}(s)$,
and so $n^{-s}$ becomes approximable by the integral $\int_n^{n+1}
x^{-s}\,dx$. This gives an efficient way to compute the tail $\sum_{n\gg
\mathfrak{q}(s)} n^{-s}$. 
Specifically, following \cite{odlyzko-schonhage-algorithm, 
 rubinstein-computational-methods}, we have, for any positive integers $N$ and
 $L_1$,
\begin{equation}
\zeta(s) = \sum_{n=1}^{N-1} n^{-s} + \frac{N^{-s}}{2} +\frac{N^{1-s}}{s-1} +
\sum_{\ell=1}^{L_1} T_{\ell,N}(s) + E_{N,L_1}(s),
\end{equation}
where $T_{\ell,N}(s)=\frac{B_{2\ell}}{(2\ell)!}N^{-s}\prod_{l=0}^{2\ell-2}
(s+l)/N$, $B_2=1/6$, $B_4=-1/30,\ldots,$ are the Bernoulli numbers, and, by the
estimate in \cite{rubinstein-computational-methods}, we have, for any $\sigma
>-(2L_1+1)$,
\begin{equation}\label{ec rem}
|E_{N,L_1}(s)| \le \frac{\zeta(2L_1)}{\pi N^{\sigma}}
\frac{|s+2L_1-1|}{\sigma+2L_1-2}
\prod_{l=0}^{2L_1-2} \frac{|s+l|}{2\pi N}.
\end{equation}
It follows from \eqref{ec rem} that, for $\sigma\ge 1/2$ say, 
one can ensure that $|E_{N,L_1}(s)|<\epsilon$ by taking $2\pi N \ge e |s+2L_1-1|$ 
and $2L_1-1 > 0.5 \log |s+2L_1-1| -\log \epsilon$.
Therefore, the remainder term in the Euler-Maclaurin summation is
easy to control, enabling very accurate computations of zeta.

Rubinstein showed~\cite{rubinstein-computational-methods} that 
one could reduce the length of the main sum in the Euler-Maclaurin
formula to $\ll \mathfrak{q}(s)^{1/2}$ terms,
but requiring $\approx \log(\mathfrak{q}(s)/\epsilon)\log(\mathfrak{q}(s))$ precision due to 
substantial cancellation that occurs, and with each term involving an incomplete
Gamma function. 
The Riemann-Siegel formula offers good control over the required precision,
and is often used in zeta computations.
The derivation of the Riemann-Siegel formula is quite involved.
One begins by expressing $\zeta(s)$ as a contour integral, then moves the contour
of integration suitably. This leads to a remainder term that requires
careful saddle-point analysis;
see  \cite[Chap. IV]{titchmarsh-book} and \cite[Chapter 7]{edwards-zeta-book}
for example. One version of the Riemann-Siegel formula 
on the critical line is the following. For $t>2\pi$,
let $a:=\sqrt{t/(2\pi)}$, $n_1:=\lfloor a\rfloor$ the
integer part of $a$, and $z:= 1-2(a-\lfloor a\rfloor)$.
Then
\begin{equation} \label{eq:rsform}
e^{i\theta(t)}\zeta(1/2+it) = 2\, \Re \left(e^{-i\theta(t)}\sum_{n=1}^{n_1} 
\frac{e^{it \log n}}{\sqrt{n}}\right) - \frac{(-1)^{n_1}}{\sqrt{a}}
\sum_{r=0}^m \frac{C_r(z)}{a^r} + R_m(t).
\end{equation}
The $C_r(z)$ can be written as a linear combination of derivatives
of the function $F(z):=\cos ((\pi/2)(z^2+3/4))(\cos(\pi z))^{-1}$ (up to the $3r$-th
derivative). 
For example, $C_0(z) = F(z)$ and $C_1(z) = F^{(3)}(z)/(12\pi^2)$,
where $F^{(3)}(z)$ is the third derivative of $F(z)$ with respect to $z$.
(Note that $F(z)$ is not periodic in $z$.) 
The general form of $C_r(z)$ can be found in Gabcke's
thesis~\cite{gabcke-thesis}.
Using formal manipulations of Dirichlet series, Berry
showed~\cite{berry-riemann-siegel} (see also \cite{berry-keating-zeta-method}) 
that the series of the correction terms $\sum_{r\ge 0} C_r(z) a^{-r}$
is divergent, and, therefore, improvement
from adding more correction terms in \eqref{eq:rsform} 
is not to continue indefinitely, instead, 
the series should be stopped at the least term for a given $t$. 
The phase $\theta(t)$ is defined by $\theta(t):=\arg[\pi^{-it/2}\Gamma(1/4+it/2)]$.
We can also define $\theta(t)$ by 
a continuous variation of $s$ in $\pi^{-s/2}\Gamma(s/2)$, starting at $s=1/2$ 
and going up vertically, which gives the formula  
$\theta(t)=(t/2)\log(t/(2\pi e))-\pi/8 + 1/(48t)+O(t^{-3})$ for large $t$.
We note that the rotation factor $e^{i\theta(t)}$ is chosen so that
$e^{i\theta(t)}\zeta(1/2+it)$
is real. Thus, one may locate non-trivial zeros of zeta
by looking for sign changes in the r.h.s. of \eqref{eq:rsform}.

As for the remainder term $R_m(t)$, we have $R_m(t) \ll t^{-(2m+3)/4}$.
Gabcke derived explicit bounds for $R_m(t)$, for $m=0,\ldots,10$ 
in his thesis~\cite{gabcke-thesis}. For example, 
for $t\geq 200$, we have
$|R_1(t)| < .053 t^{-5/4}$, $|R_4(t)| <  0.017t^{-11/4}$, 
and $|R_{10}(t)| < 25966 t^{-23/4}$.
While Gabcke's estimates are sufficient for most 
applications, they do not allow for very high 
accuracy for relatively small $t$,
such as required when computing zeta zeros to many digits in order 
to test their linear independence. (Recently, very good bounds have been
derived in \cite{reyna}.)
A source of the difficulty towards explicit estimates 
of $R_m(t)$ is that the main sum of the Riemann-Siegel
formula has a sharp cut-off (dictated by the location of the saddle-point), 
which complicates the analysis of the remainder
term significantly. The analysis
is much simplified by using a smoothing function.
Indeed, Turing had 
proposed~\cite{turing-zeta-method} a type of smoothed formula for
computing zeta in the intermediate range where
$t$ is neither so small that the Euler-Maclaurin summation can be 
used nor large enough for the Riemann-Siegel
asymptotic formula.\footnote{It is worth mentioning that Theorem~\ref{zeta alg}
is useful in such a range, in order to carry out
high precision computations.}
Rubinstein provides~\cite{rubinstein-computational-methods}
the following smoothed formula, which 
has a main sum of length $\mathfrak{q}(s)^{1/2+o_{\epsilon}(1)}$,
and which can be generalized to a fairly large class of $L$-functions.
\begin{equation} \label{smoothed approx}
\pi^{-s/2}\Gamma(s/2)\zeta(s)\delta^{-s}
= -\frac{1}{s}-\frac{\delta^{-1}}{1-s}
+\sum_{n=1}^{\infty} G(s/2,\pi n^2\delta^2) 
+\delta^{-1}\sum_{n=1}^{\infty} G((1-s)/2,\pi n^2/\delta^2),
\end{equation}
where $G(z,w)$ is a smoothing function that can be expressed in terms of
the incomplete Gamma function $\Gamma(z,w)$,
\label{Gzw}
$G(z,w) := w^{-z}\Gamma(z,w) = \int_1^{\infty} e^{-wx}x^{z-1}\,dx$, $\Re(w)
> 0$,
and $\delta$ is a complex parameter of modulus one, with a simple dependence on
$t$, such that $|\Im (\log \delta)| \in (-\pi,\pi]$
and $\Im(\log \delta)$ tends to $\sgn(t) \pi/4$ for large $t$. In explicit form,
$\delta = \exp(i\sgn(t) (\pi/4 - \theta))$, where 
$\theta = \pi/4$ if $|t| \le 2c/\pi$, $\theta = c/|2t|$ if $|t|>2c/\pi$,
and $c>0$ is a free parameter that we can optimize.
In particular, $\delta^{-s}$ is chosen to cancel out the exponential decay
in $\Gamma(s/2)$ as $t$ gets large on the l.h.s\ of \eqref{smoothed
approx}, ensuring that the l.h.s.\ is $\gg |s|^{(\sigma-1)/2}|\zeta(s)| e^{-c}$
for large $t$.
Although the series in \eqref{smoothed approx} are infinite, the weights
$G(z,w)$ decay exponentially fast when $\Re(w) \gg 1$. Specifically,
following \cite{rubinstein-computational-methods}, we have
for $\Re(w)>0$ and $\Re(z)\le 1$ that $|G(z,w)| < e^{-\Re(w)}/\Re(w)$.
So, for $|t| > 2c/\pi$ and $\sigma \in [0,1]$ say,
we have $\Re(\pi n^2\delta^2)=\Re(\pi n^2/\delta^2) = \pi n^2 \cos(\pi/2 -
c/|t|) > \pi n^2 c/|2t|$, where we used
the inequality $\cos(\pi/2-x)\ge x/2$ for $0\le x\le 1$. 
Therefore, the series can be truncated 
after $M$ terms with truncation error $< 4|t|/(\pi c) \sum_{n\ge M} n^{-2} e^{-\pi n^2
c/|2t|}$. So to ensure that the truncation error is $<\epsilon$, it
certainly suffices to take $M > \sqrt{|2t|/(\pi c) \log (|4t|/(\epsilon c))}$.
Once the series is truncated, it can be evaluated term by term to give
a numerical approximation of $\zeta(\sigma+it)$ for $|t|>2c/\pi$.
The number of terms in the resulting main sum (i.e.\ truncated series) is roughly equal to
$\sqrt{\mathfrak{q}(s)\log(\mathfrak{q}(s)/\epsilon)}$.
The terms in the main sum are more complicated than in the
Riemann-Siegel formula 
since each term involves the smoothing function $G(z,w)$.

In the case of Dirichlet $L$-functions, 
Davies~\cite{davies-dirichlet-l-functions},
Deuring~\cite{deuring-dirichlet-l-functions}, Lavrik~\cite{lavrik}, and
others had developed Riemann-Siegel type formulas for $L(1/2+it,\chi)$,
where $\chi$ is a primitive character mod $q$ and $t\gg 1$.
Such formulas, whose general form was already considered by
Siegel~\cite{siegel-dirichlet-l-functions},
require the numerical evaluation of a main sum of length $q\lfloor \sqrt{t/(2\pi
q)}\rfloor\approx \mathfrak{q}(\chi,s)^{1/2}$ terms, 
where each term is of the form $\chi(n) n^{-1/2}\exp(it\log n)$.
Unfortunately, however, it does not seem that we have an analogue of Gabcke's explicit estimate
 for the remainder terms in such formulas. 
And it is not
clear how to obtain a posteriori error estimate either.
Therefore, we are not prepared to find the accuracy of
the numerics resulting from these formulas explicitly.
Still, if one is willing to live with a much longer main sum,
consisting of about $\mathfrak{q}(\chi,s)$ terms,
then one can keep the simplicity of an unsmoothed main sum while
having an explicit estimate for the remainder term.
The basic idea is well-known, and was implemented carefully
by Rumely~\cite{rumely-dirichlet-l-functions}.
Essentially, one uses the periodicity of $\chi$ to
write $L(1/2+it,\chi)$ as a linear combination of about $\mathfrak{q}(\chi,0)$
Hurwitz zeta functions, then one approximates each Hurwitz zeta function
using the Euler-Maclaurin
summation formula.
However, since the Euler-Maclaurin formula requires
a main sum of length about $\mathfrak{q}(s)$, the cost of this method is prohibitive in comparison
with a Riemann-Siegel approach with explicit remainder.
In view of this, one typically uses a smoothing function to accelerate the
convergence.
Such formulae (see \cite{rubinstein-computational-methods}) are applicable even for small $t$ and
have a main sum of length $\mathfrak{q}(\chi,s)^{1/2+o_{\epsilon}(1)}$ terms, 
where each term involves the computation of a smoothing function. 

\section{Proofs of Theorems \ref{zeta alg} \& \ref{dirichlet alg}}\label{proofs}
We first prove Theorem~\ref{zeta alg}. 
The proof of Theorem~\ref{dirichlet alg}
will be similar, but will additionally require 
a specialization of the Postnikov character formula, lemma~\ref{postnikov}.
Recall that we choose integers $u_0\ge 1$, $v_0\ge u_0$, $M\ge v_0$,
and we construct the sequences $K_r= \min\{\lceil v_r/u_0\rceil, M-v_r\}$ 
and $v_{r+1}= v_r + K_r$ for $0\le r\le R$, where 
$R:=R(v_0,u_0,M)$ is the smallest integer such that $v_{R+1} = M$. 
\begin{lemma} \label{R bound}
$R = R(v_0,u_0,M) < 2u_0 \log(M/v_0)+1$.
\end{lemma}
\begin{proof}
For $r<R$, we have $v_{r+1} = v_r + K_r \ge v_r(1+1/u_0)$, and so
by induction $v_{r+1} \ge v_0(1+1/u_0)^r$. If $R>0$, then 
taking $r=R-1$ and noting that $v_R < M$, we obtain 
$R < \log(M/v_0)/\log(1+1/u_0) +1 \le 2u_0 \log(M/v_0)+1$, 
where we used the inequality $\log(1+x) \ge x/2$ for $0\le x<1$.
If $R=0$, then clearly the last bound still holds.
\end{proof}
\begin{lemma}\label{Kr bound}
Let $s=\sigma+it$, $\sigma\ge 0$. 
Using the same notation for $K_r$, $v_r$, and $R$, we have
\begin{displaymath}
\sum_{0\le r\le R} g_{K_r}(-\sigma/v_r) v_r^{-\sigma} 
\le v_0^{-\sigma}+ \left\{\begin{array}{ll}
\frac{M^{1-\sigma}-v_0^{1-\sigma}}{1-\sigma}, &\sigma\ne 1, \\
\log(M/v_0), &\sigma = 1.\\
 \end{array}\right.
\end{displaymath}
\end{lemma}
\begin{proof}
For $k<v$, we have $\log(1+k/v) = k/v - k^2/(2v^2) +\cdots \le k/v$.
Thus, $e^{-\sigma k/v} \le (1+k/v)^{-\sigma}$.
Hence, $g_K(-\sigma/v) v^{-\sigma}\le v^{-\sigma}\sum_{0\le k<K} (1+k/v)^{-\sigma} 
= \sum_{0\le k<K} (v+k)^{-\sigma}$.
So $\sum_{0\le r\le R} g_{K_r}(-\sigma/v_r) v_r^{-\sigma}
\le \sum_{v_0\le n<M} n^{-\sigma} \le v_0^{-\sigma} + \int_{v_0}^M x^{-\sigma}\,dx$.
The lemma follows on evaluating the integral.
\end{proof}
\begin{lemma}\label{h sum}
Let $s=\sigma+it$, $\sigma\ge 0$. 
For  any integers $v\ge u\ge 2\max\{6,\sqrt{|s|},\sigma\}$, $K\ge 1$, and $m\ge 0$, 
such that $(K-1)/v \le 1/u$, we have
\begin{equation}\label{geom form}
\sum_{0\le k<K} e^{-s\log(1+k/v)} = 
\sum_{j=0}^m c_j(s) \frac{g_K^{(j)}(-s/v)}{v^j} + \mathcal{E}_m(s,v,K),
\end{equation}
$c_j(s) = \frac{f_s^{(j)}(0)}{j!}$,
$|\mathcal{E}_m(s,v,K)| \le
\epsilon_m(s,u)\min\{g_K(-\sigma/v),|\csc(t/(2v))|e^{-\sigma(K-1)/v}\}$,
and $\epsilon_m(s,u)$ is defined in \eqref{eps def}.
\end{lemma}
\begin{proof}
We have $e^{-s\log(1+k/v)} = e^{-sk/v} f_s(k/v)$.
The function $f_s(z)$ is analytic in $|z| <1$.
Taking the branch of the logarithm 
determined by $f_s(0)=1$, 
we have $f_s(z) = e^{-s\log(1+z)+sz} = 
e^{sz^2/2 - sz^3/3+\cdots}$ for $|z|<1$. 
We expand $f_s(z)$ into a power series 
$1+\cdots+c_m(s) z^m + \cdots$.
By definition, we have $\mathcal{E}_m(s,v,K) = \sum_{0\le k < K} e^{-sk/v} 
\sum_{j > m} c_j(s) (k/v)^j$.
So, interchanging the order of summation in $j$ and $k$,
we obtain
$|\mathcal{E}_m(s,v,K)| \le \sum_{j>m} |c_j(s)| 
|\sum_{0\le k<K} (k/v)^j e^{-sk/v} |$.
We note that the function $x^j e^{-\sigma x}$ is increasing with $x$ if 
 $0\le x< j/\sigma$. So, if $0\le k < jv/\sigma$, then 
$(k/v)^je^{-\sigma k/v}$ increases with $k$. This last condition, $k<jv/\sigma$,
is satisfied because, by hypothesis, $j>m\ge 0$, so $j\ge 1$, and 
$k/v\le 1/u< 1/\sigma$.
Thus, it follows by partial summation that
\begin{equation}
|\mathcal{E}_m(s,v,K)| \le e^{-\sigma(K-1)/v} \max_{x\in[0,K]} |\sum_{x\le k<K} e^{-itk/v}| 
\sum_{j>m} |c_j(s)| (K-1)^j/v^j.
\end{equation}
Executing the summation in the geometric sum, we see that 
it is bounded by $|\csc(t/(2v))|$.
Also, by a trivial estimate, 
$|\sum_{0\le k<K}e^{-sk/v} |\le g_K(-\sigma/v)$.
Thus,
\begin{equation}
|\mathcal{E}_m(s,v,K)|\le 
\min\{g_K(-\sigma/v),|\csc(t/(2v))|e^{-\sigma(K-1)/v}\} \sum_{j>m} |c_j(s)| \frac{(K-1)^j}{v^j}.
\end{equation}
We bound $c_j(s)$
by a standard application of Cauchy's theorem using 
a circle around the origin.
We have $2\pi |c_j(s)| \le  |\int_{|z|=c}
f_s(z)/z^{j+1}\,dz | \le  2\pi c^{-j}
e^{|s|c^2/2+\cdots}$, $c\in (0,1)$. 
If $0< j\le |s|/4$, let
$c =\sqrt{j/|s|} \le 1/2$. So
$|s|c^2/2+\cdots \le |s|c^2\sum_{r=2}^{\infty}
c^{r-2}/r \le \alpha j$, where $\alpha:=\sum_{r=2}^{\infty}
(1/2)^{r-2}/r= -2+4\log 2<0.78$. We conclude that
$|c_j(s)| \le  |s|^{j/2} j^{-j/2} e^{\alpha j}$ for
$0< j\le |s|/4$.
Also, for any $j\ge 0$, we may choose $c=1/2$. 
So we have $|c_j(s)| \le 2^j e^{\alpha|s|/4}$ for each $j\ge 0$. 

Since $(K-1)/v\le 1/u$, by hypothesis, we have by the estimate
for $c_j(s)$, and assuming that $m\le |s|/4$, that
\begin{equation}
 \sum_{j>m} |c_j(s)| \frac{(K-1)^j}{v^j} \le \sum_{m<j \le |s|/4} 
|s|^{j/2} u^{-j}j^{-j/2} e^{\alpha j} +
\sum_{j > |s|/4} u^{-j} 2^j e^{\alpha |s|/4}.
\end{equation}
If $|s|/4$ is not an integer, then 
$\sum_{j > |s|/4} u^{-j}2^je^{\alpha|s|/4} 
\le 0.2e^{\alpha} (u/2)^{-\lfloor |s|/4\rfloor}e^{\alpha\lfloor|s|/4\rfloor}$,
where we used $\sum_{\ell>0} (2/u)^{\ell}\le 0.2$
and $u\ge 12$. 
Since this is at most $0.2e^{\alpha}<0.44$ times the last
term in first sum on the r.h.s.\ above, 
we obtain the estimate
\begin{equation}
 \sum_{j>m} |c_j(s)| \frac{(K-1)^j}{v^j} 
\le 1.44\sum_{j>m} |s|^{j/2} u^{-j} j^{-j/2} e^{\alpha j}.
\end{equation}
Now, for $\ell \ge 0$, $(m+1+\ell)^{-(m+1+\ell)/2} 
\le (m+1)^{-(m+1)/2} (1+\ell)^{-\ell/2}$. 
Therefore, 
\begin{equation}
\sum_{j>m} |c_j(s)| \frac{(K-1)^j}{v^j} \le 
\frac{3.5\, e^{\alpha(m+1)}}{(m+1)^{(m+1)/2}}\frac{|s|^{(m+1)/2}}{u^{m+1}}
<  \epsilon_m(s,u),
\end{equation}
where we used $u\ge 2\sqrt{|s|}$ and 
$\sum_{\ell=0}^{\infty} \frac{|s|^{\ell/2}u^{-\ell}e^{\alpha \ell}}{(1+\ell)^{\ell/2}}\le 2.42$,
so $(2.42)(1.44)<3.5$.
If $|s|/4$ is an integer, on the other hand,
then the same bound holds (with an even better constant).
It remains to consider the case when $m>|s|/4$. Here, we have 
$\sum_{j>m} |c_j(s)| (K-1)^j/v^j \le 
\sum_{j > m} u^{-j} 2^j e^{\alpha |s|/4}\le
2^m e^{\alpha|s|/4}/u^m$.
Therefore, $\sum_{j>m} |c_j(s)|(K-1)^j/v^j \le \epsilon_m(s,u)$.
Put together, we arrive at the claimed bound on $\mathcal{E}_m(s,v,K)$.
To complete the proof of the lemma, notice that
\begin{equation}\label{h eq}
\sum_{0\le k<K} e^{-s\log(1+k/v)}
= \sum_{0\le k <K} \sum_{j=0}^m c_j(s) (k/v)^j e^{-sk/v} + 
\mathcal{E}_m(s,v,K).
\end{equation}
So the formula \eqref{geom form} follows on interchanging the order of
the double sum.
\end{proof}
\begin{lemma}\label{postnikov}
Let $\chi\bmod{p^a}$ be a Dirichlet character, where $p$ is a prime,
and let $b=\lceil a/2\rceil$.
Then there exists an integer $L\bmod{p^{a-b}}$, depending on
$\chi$, $p$, $a$, and $b$ only (so independent of $x$),  such that
$\chi(1+p^b x)=e^{2\pi i L  x/p^{a-b}}$
for all $x\in \mathbb{Z}$. 
\end{lemma}
\begin{proof}
The proof is similar to that of \cite[Lemma 4.2]{hiary-char-sums}, 
but we still give it here for completeness.
Let $H$ be the subgroup in $\left(\mathbb{Z}/p^a\mathbb{Z}\right)^*$
consisting of the residue classes congruent to $1\bmod{p^b}$, so
$H$ has size $|H|=p^{a-b}$.
We identify the elements of $H$ with the set of integers
$\{1+p^b x\, |\, 0\le x <p^{a-b}\}$.
Consider the function $\psi: H\to \mathbb{C}$, defined by
$\psi(1+p^b x):=e^{2\pi i x/p^{a-b}}$
By our choice of $b=\lceil a/2\rceil$, we have  $p^{b}\equiv 0\bmod{p^{a-b}}$.
Therefore, $\psi((1+p^b x)(1+p^b y))=\psi(1+p^b x)\psi(1+p^b y)$
for all $x,y,\in\mathbb{Z}$, meaning that $\psi$ is multiplicative. 
Also, $\psi$ is not identically zero; e.g.\ $\psi(1) = 1$. 
Therefore, $\psi$  is a character of $H$.
Moreover, the values $\psi(1+p^b)^u = e^{2\pi i u /p^{a-b}}$,
$0\le u < p^{a-b}$, are all distinct.
In particular, $\psi$ has order $p^{a-b}$, which is the same as the order of
$H$. So $\psi$ generates the full character group of $H$.
Since $\left.\chi\right|_H$ is a character of $H$,
then $\left.\chi\right|_{H} \equiv \psi^L$ for some $L \bmod{p^{a-b}}$. 
To find $L$, we calculate $\chi(1+p^b)$, then use
 the relation $\chi(1+p^b)=e^{2\pi i L /p^{a-b}}$.
\end{proof}
\begin{proof}[Proof of Theorem~\ref{zeta alg}]
We divide the main sum in \eqref{zeta trunc} according to the positions 
of $v_r$ as follows:
$\sum_{1\le n<v_0} n^{-s} + \sum_{0\le r\le R} v_r^{-s} \sum_{0\le k<K_r}
e^{-s\log(1+k/v_r)}$.
Note that $K_r = \lceil
v_r/u_0\rceil \le v_r/u_0+1$ for $r<R$, and $K_R \le v_R/u_0+1$. 
So $(K_r-1)/v_r \le 1/u_0$ throughout $0\le r\le R$.
Thus, the conditions for lemma~\ref{h sum} are satisfied and 
we can apply it to each block 
$\sum_{0\le k<K_r} e^{-s\log(1+k/v_r)}$. This yields
$\mathcal{T}_{M,m}(s,u_0,v_0)=\sum_{0\le r\le R} v_r^{-s} \mathcal{E}_m(s,v_r,K_r)$.
And using the estimate for $\mathcal{E}_m(s,v,K)$ in lemma~\ref{h sum}
yields the required bound on $\mathcal{T}_{M,m}(s,u_0,v_0)$.
\end{proof}
\begin{proof}[Proof of Lemma~\ref{gKr lemma}]
This follows from from the definitions and lemma~\ref{postnikov}:
\begin{equation}
\begin{split}
g_K(z,\chi,v) &= \sum_{d=0}^{p^b-1} \sum_{0\le k<H_d}
e^{z(d+p^bk)}\chi(v+d+p^bk)\\
&= \sum_{d=0}^{p^b-1} \delta_{\gcd(v+d,p)=1} \chi(v+d) e^{zd}\sum_{0\le k<H_d}
e^{p^bzk}\chi(1+p^b\overline{v+d}k)\\
&= \sum_{d=0}^{p^b-1} \delta_{\gcd(v+d,p)=1} \chi(v+d) e^{zd}\sum_{0\le k<H_d}
e^{p^bzk + 2\pi iL \overline{v+d}k/p^{a-b}}\\
&= \sum_{d=0}^{p^b-1} \chi(v+d) e^{zd} g_{H_d}(p^bz+iw_d).
\end{split}
\end{equation}
\end{proof}
\begin{lemma}\label{h chi sum}
Given $s=\sigma+it$, $\sigma\ge 0$, and a Dirichlet character $\chi\bmod{p^a}$
with $p$ a prime, let $b=\lceil a/2\rceil$. 
Then for  any integers $v\ge u\ge 2\max\{6,\sqrt{|s|},\sigma\}$, $K\ge 1$, and $m\ge 0$, 
such that $(K-1)/v \le 1/u$, we have
\begin{equation}\label{geom chi form}
\sum_{0\le k<K} \chi(v+k) e^{-s\log(1+k/v)} = 
\sum_{j=0}^m c_j(s) \frac{g_K^{(j)}(-s/v,\chi,v)}{v^j} + \mathcal{E}_m(s,\chi,v,K),
\end{equation}
where $c_j(s) = \frac{f_s^{(j)}(0)}{j!}$, and, with $H_d=\lceil
(K-d)/p^b\rceil$, we have 
\begin{equation}
|\mathcal{E}_m(s,\chi,v,K)| \le 
\epsilon_m(s,u)\sum_{d=0}^{p^b-1}
\delta_{\gcd(v+d,p)=1} 
\min\{g_{H_d}(-p^b\sigma/v),|\csc(w_d/2-p^bt/(2v))|\}.
\end{equation}
The $\epsilon_m(s,u)$ is defined in \eqref{eps def}.
\end{lemma}
\begin{proof}
Proceeding in the same way as in Theorem~\ref{zeta alg}
and lemma~\ref{gKr lemma}, we arrive at
\begin{equation}
\begin{split}
&\mathcal{E}_m(s,\chi,v,K) =
\sum_{0\le k<K} \chi(v+k) e^{-sk/v} \sum_{j>m} c_j(s) (k/v)^j\\
&= \sum_{j>m}c_j(s) \sum_{0\le k<K} \chi(v+k) e^{-sk/v} (k/v)^j\\
&= \sum_{j>m}c_j(s) \sum_{d=0}^{p^b-1} 
\sum_{0\le k<H_d} \chi(v+d+p^bk) e^{-s(d+p^bk)/v} ((d+p^bk)/v)^j\\
&= \sum_{j>m}c_j(s) \sum_{d=0}^{p^b-1} \chi(v+d)
\sum_{\ell=0}^j \binom{j}{\ell} d^{j-\ell} v^{\ell-j} 
e^{-sd/v} \sum_{0\le k<H_d} (p^bk/v)^{\ell} e^{(-p^bs/v+iw_d)k}.
\end{split}
\end{equation}
Therefore, using partial summation, as in the proof of lemma~\ref{h sum}, we obtain 
\begin{equation}
\begin{split}
&|\mathcal{E}_m(s,\chi,v,K)| \\
&\le 
\sum_{j>m} |c_j(s)| \sum_{d=0}^{p^b-1}\delta_{\gcd(v+d,p)=1}  
\sum_{\ell=0}^j \binom{j}{\ell} (d/v)^{j-\ell} e^{-\sigma d/v} 
 |\sum_{0\le k<H_d} (p^bk/v)^{\ell} e^{(-p^bs/v+iw_d)k}|\\
&\le \sum_{j>m} |c_j(s)| \sum_{d=0}^{p^b-1} \delta_{\gcd(v+d,p)=1} 
((d+p^b(H_d-1))/v)^j
\max_{x\in[0,H_d]} |\sum_{x\le k<H_d} e^{i(-p^bt/v+w_d)k}|\\
&\le \sum_{j>m} |c_j(s)|(K-1)^j/v^j
\sum_{d=0}^{p^b-1}\delta_{\gcd(v_r+d,p)=1}
\max_{x\in[0,H_d]} |\sum_{x\le k<H_d} e^{i(-p^bt/v+w_d)k}|\\
&\le \epsilon_m(s,u) \sum_{d=0}^{p^b-1}
\delta_{\gcd(v_r+d,p)=1}
|\csc(w_d/2-p^bt/(2v))|.
\end{split}
\end{equation}
Combined with the trivial estimate, this yields the lemma.
\end{proof}
\begin{proof}[Proof of Theorem~\ref{dirichlet alg}]
We divide the main sum in \eqref{dirichlet trunc} according to the positions
of $v_r$ as before:
$\sum_{1\le n<v_0} \chi(n)n^{-s} + \sum_{0\le r\le R} v_r^{-s} \sum_{0\le k<K_r}
\chi(v+r)e^{-s\log(1+k/v_r)}$.
We apply lemmas \ref{h chi sum} and \ref{gKr lemma}
to the sum over $k$. This yields the result with
$\mathcal{T}_{M,m}(s,\chi,u_0,v_0)=\sum_{0\le r\le R} v_r^{-s}
\mathcal{E}_m(s,\chi,v_r,K_r)$.
By the estimate for $\mathcal{E}_m(s,\chi,v,K)$ in lemma~\ref{h chi sum}
we obtain the desired bound on $\mathcal{T}_{M,m}(s,\chi, u_0,v_0)$.
\end{proof}

\section{Computing $\displaystyle \frac{f_s^{(j)}(0)}{j!}\frac{g_K^{(j)}(z)}{v^j}$
for $0\le j\le m$} \label{terms comp}
One can choose the parameters
in Theorems \ref{zeta alg} \& \ref{dirichlet alg} so that 
one can achieve moderate accuracy with $m\le 8$, say.
So, in general,
computing $(f_s^{(j)}(0)/j!)(g_K^{(j)}(z)/v^j)$ will be quite
easy, and can be done using closed-form
formulas to evaluate the geometric sum. The methods that we present below
are intended for when $j$ is large, but they can be used for any $j\ge 0$.
In our application (Theorems \ref{zeta alg} \& \ref{dirichlet alg}),
we have $(K-1)/v \le 1/u_0\le 1/(2\sqrt{\mathfrak{q}(s)})$, 
and $-1/2\le \Re(z)\le 0$. So we will assume that this holds throughout. 

We recall that 
$f_s(z) = e^{sz^2/2-sz^3/3+\cdots} = \sum_{j\ge 0} \frac{f^{(j)}_s(0)}{j!} z^j$ for
$|z|<1$. For example, 
\begin{equation}
\begin{split}
&f_s(0) = 1,\qquad f^{(1)}_s(0) = 0,\qquad f^{(2)}_s(0) = s,\qquad f^{(3)}_s(0)
= -2s,\\
&f^{(4)}_s(0) = 3s(2+s),\qquad f^{(5)}_s(0) = -4s(6+5s),\\
& f^{(6)}_s(0) = 5s(24+26s+3s^2),\qquad f^{(7)}_s(0) = -6s(120+154s+35s^2),\\
&f^{(8)}_s(0) = 7s(720+1044s+340s^2+15s^3),\quad \ldots.
\end{split}
\end{equation}
To find $f^{(j)}_s(0)$ in general, let 
$q(z):=\sum_{\alpha\ge 2} (-1)^{\alpha}z^{\alpha}/\alpha$,
so $f_s^{(j)}(z) = Q_{s,j}(z) e^{sq(z)}$
for some $Q_{s,j}(z) = \sum_{l\ge0} w_{j,l}(s) z^l$ that 
satisfies the recursion $Q_{s,0}(z) := 1$ and  
$Q_{s,j+1}(z) = \frac{d}{dz} Q_{s,j}(z) +s Q_{s,j}(z)\frac{d}{dz}q(z)$. 
Therefore, $w_{j,0}(s) = f^{(j)}_s(0)$, 
$w_{0,0}(s) = 1$, $w_{0,l}(s)=0$ for $l>0$, and 
$w_{j+1,l}(s) = (l+1)w_{j,l+1}(s) - 
s\sum_{\alpha =1}^l (-1)^{\alpha} w_{j,l-\alpha}(s)$. 
Using this recursion, one can find all of $f^{(j)}_s(0)=w_{j,0}(s)$ for 
$0\le j\le m$ in about $(m+1)^2$ steps. In carrying out
the recursion, one may treat $s$ symbolically, so 
$w_{j,0}(s)$ is viewed as a polynomial in $s$ and the recursion is finding
the coefficients of this polynomial. 
In fact, it follows from the recursion that, more generally, $w_{j,l}(s)$ 
 is a polynomial in $s$ of degree  $\le \min\{(j+l)/2,j\}$.
So we may write $w_{j,l}(s)=\sum_{0\le \eta\le j/2} \beta_{j,l,\eta} s^{\eta}$. 
Also, $\beta_{0,0,0} = 1$, $\beta_{0,0,\eta}=0$ for $\eta>0$,
$\beta_{0,l,\eta}=0$ for $l>0$, 
and we have $\beta_{j+1,l,\eta} = (l+1)\beta_{j,l+1,\eta} - 
\sum_{\alpha=1}^l (-1)^{\alpha} \beta_{j,l-\alpha,\eta-1}$.
Therefore, using induction, we obtain the bound 
$|\beta_{j,l,\eta}|\le (j+l)!2^{l+1}/l!$.
In particular, $|\beta_{j,0,\eta}|/j!\le 2$. 
Thus, the number of bits needed to represent $|\beta_{j,0,\eta}|/j!$, 
and hence to compute $f_s^{(j)}(0)/j!$ as a polynomial in $s$, 
to a given precision, is also well-controlled.

As for computing $g_K^{(j)}(z)$, one can use the formula
$g_K^{(j)}(z) = \sum_{\ell=0}^j \binom{j}{\ell} w^{(j-\ell)}(z)y^{(\ell)}(z)$, 
where $w(z) := e^{Kz}-1$ and $y(z):=(e^z-1)^{-1}$. 
So for $z\not\in 2\pi i \mathbb{Z}$ we have
\begin{equation}\label{fg comp}
\begin{split}
\frac{f_s^{(j)}(0)}{j!} \frac{g_K^{(j)}(z)}{v^j} &= 
\frac{f_s^{(j)}(0)}{j!}\frac{2^j(K-1)^j}{v^j} \frac{g_K^{(j)}(z)}{2^j(K-1)^j} \\ 
&=\frac{f_s^{(j)}(0)}{j!} \frac{2^j(K-1)^j}{v^j}\left( e^{Kz} 
\sum_{\ell=0}^j \frac{1}{2^j}\binom{j}{\ell}
\frac{y^{(\ell)}(z)}{(K-1)^{\ell}} - \frac{y^{(j)}(z)}{(K-1)^j}\right).
\end{split}
\end{equation}
The factor $2^{-j}$ is inserted inside the sum in \eqref{fg comp} in order to
to control the size of the binomial coefficient $\binom{j}{l}\le 2^j$.
By hypothesis, $(K-1)/v\le 1/u_0\le 1/(2\sqrt{\mathfrak{q}(s)})$.  
So, recalling that $f_s^{(j)}(0) = \sum_{0\le \eta \le j/2} \beta_{j,0,\eta}s^{\eta}$,
$|\beta_{j,0,\eta}|\le 2(j!)$, and $\mathfrak{q}(s)\ge 3$, we obtain
$|f_s^{(j)}(0) 2^j(K-1)^j|/(j!v^j)\le 5$.
In particular, the number of bits
needed to represent the outside factor in \eqref{fg comp} is well-controlled, and 
we may focus on computing the sum enclosed in parentheses.

To that end, we consider the computation of $y^{(\ell)}(z)/(K-1)^{\ell}$ in
\eqref{fg comp}.
If $\ell$ is small, this can be done by directly
differentiating $y(z)$, but this is not a practical method if $\ell$ is large. 
Instead, we note that
$zy(z)=z/(e^z-1)$ is the exponential generating function for the Bernoulli numbers,
specifically,
\begin{equation}
y(z)= \frac{1}{z} -\frac{1}{2} + \sum_{l=1}^{\infty} \frac{B_{2l}}{(2l)!}
z^{2l-1},\quad 0<|z|<2\pi.
\end{equation}
Therefore, for $\ell >0$, 
\begin{equation}
y^{(\ell)}(z) = \frac{(-1)^{\ell} \ell!}{z^{\ell+1}} + 
\sum_{l=\lceil (\ell+1)/2\rceil}^{\infty} 
\frac{B_{2l}}{2l} \frac{z^{2l-\ell-1}}{(2l-\ell-1)!},\quad 0<|z|<2\pi.
\end{equation}
Using the periodicity of $e^z$, and our assumption on $z$,
 we can ensure that the argument given 
 to $g_K^{(j)}(z)$ satisfies $|z| < 3\pi/2$.
Thus, the above formulas will 
suffice to compute $y^{(\ell)}(z)/(K-1)^{\ell}$ provided
that $|z|$ is sufficiently bounded away from
$0$, say $|z|> (m+1)/(K-1)$.
For such $z$, and assuming that $K> 2\pi (m+1)$
 (otherwise, we may compute $g_K^{(j)}(z)$ by direct summation in $\ll m+1$
 steps),
we obtain that $y^{(\ell)}(z)/(K-1)^{\ell}$ is bounded by a constant, and so its size
is well-controlled.
Thus, the only remaining case is when $|z|<(m+1)/(K-1)$, with $K>2\pi (m+1)$.
In this case, we use the Euler-Maclaurin summation.
To this end, let $h_{j,z}(x) :=x^je^{z x}$.
Then $g_K^{(j)}(z) = \sum_{0\le k < K} h_{j,z}(k)$.
Note that, using the periodicity of $e^{zk}$ and 
conjugating if necessary, we may assume that $0\le \Im(z) \le \pi$. 
By the Euler-Maclaurin formula (see \cite{rubinstein-computational-methods}),
we have
\begin{equation}\label{em geometric}
\begin{split}
g_K^{(j)}(z)  &= \int_0^{K-1} h_{j,z}(x)\,dx 
+ \sum_{\ell=1}^L \frac{B_{2\ell}}{(2\ell) !}(h_{j,z}^{(2\ell-1)}(K-1)
-h_{j,z}^{(2\ell-1)}(0)) \\
&+ \frac{1}{2}(h_{j,z}(K-1)+ h_{j,z}(0))+\mathcal{E}_{K,j,z,L}, 
\end{split}
\end{equation}
where $h_{j,z}^{(2\ell-1)}(x)$ is
the $(2\ell-1)$-st derivative of $h_{j,z}(x)$ with respect to $x$,
and the remainder term 
$\mathcal{E}_{K,j,z,L} = (-1/(2L)!)\int_0^{K-1}
B_{2L}(\{x\})h_{j,z}^{(2L)}(x)\,dx$, 
where $B_{2L}(x)$ is the $2L$-th Bernoulli polynomial (e.g.\ $B_2(x) = x^2-x+1/6$), 
and $\{x\}$ is the factional part of $x$. 
Now, $h_{j,z}^{(2\ell-1)}(x)= \sum_{l=0}^{2\ell-1} \binom{2\ell-1}{l}
(\frac{d^l}{dx^l} x^j) (\frac{d^{2\ell-l-1}}{dx^{2\ell-l-1}} e^{zx})$. 
Thus, we have
\begin{equation}
h_{j,z}^{(2\ell-1)}(x) = e^{zx} \sum_{l=0}^{\min\{2\ell-1,j\}} \binom{2\ell-1}{l}
\frac{j!}{(j-l)!} x^{j-l} z^{2\ell-l-1}.
\end{equation}
Also, from the Fourier expansion for $B_{2L}(\{x\})$ (see
\cite{rubinstein-computational-methods}), $|B_{2L}(\{x\})|\le
4\frac{(2L)!}{(2\pi)^{2L}}$. Therefore, 
since $\Re(z)\le 0$, we deduce that 
$|\mathcal{E}_{K,j,z,L}|/(K-1)^j \le 4(K-1)(2\pi)^{-2L}$,
which decays exponentially with $L$.

As for the main term $\int_0^{K-1} h_{j,z}(x)\,dx$
in formula~\eqref{em geometric}, its computation
does not present any difficulty since $|z|<m/K$ 
(so $z$ is small).
For example, one can split
the interval of integration into $m+1$ consecutive subintervals
of equal length, then, after a suitable change of variable, 
apply Taylor expansions to the integrand in each
subinterval, which reduces the problem to integrating polynomials. 
Alternatively, one can use a numerical quadrature rule.

\section{A convexity bound}\label{convexity-bounds}
We will use the following well-spacing lemma to prove corollary~\ref{zeta bound}.
\begin{lemma}\label{csc sum lemma}
Let $\{x_n, n=0,1,\ldots\}$ be a set of real numbers. Suppose there exists a
positive
integer $Q$ such that $\min_{n\ne n'}|x_n-x_{n'}| \ge 1/(2Q)$.
Then, for any $y\ge x$ and any $P\ge 1$, we have
\begin{equation}
\sum_{x_n \in [x,y]} 
\min\{P,|\csc(\pi x_n)|\} \le (1+ \lfloor
y-x\rfloor)(2(A+1)P+2Q\log(Q/A)),
\end{equation}
where $A$ is any positive integer that satisfies $A \le Q/P$.
\end{lemma}
\begin{proof}
Since $|x_n-x_{n'}| \ge 1/(2Q)$ for $n\ne n'$, then
for any integer $k$ we have
$\sum_{x_n\in[k-1/2,k+1/2]} \min\{P,|\csc(\pi
x_n)|\} \le 
P+\sum_{|l|\le Q} \min\{P,|\csc(\pi l/(2Q))|\}=:*$.
Using the inequality $|\sin(\pi \alpha)|\ge 2|\alpha|$, $-1/2\le \alpha\le 1/2$, 
we obtain that $*\le 2(A+1)P+ \sum_{A< |l|\le Q} Q/|l|$.
Combined with the inequality $\sum_{A<l\le Q} 1/l \le \log(Q/A)$,
this gives $*\le 2(A+1)P+2Q\log(Q/A)$.
Since the interval $[x,y]$ 
contains $\le 1+ \lfloor y-x\rfloor$ integers, the lemma follows.
\end{proof}
The bound that we obtain in corollary~\ref{zeta bound} for zeta is, of course, 
superseded by the bound that one can obtain from the Riemann-Siegel
formula. Nevertheless, it illustrates that 
Theorem~\ref{zeta alg}  yields a convexity bound
of similar strength to the Riemann-Siegel formula,
up to a constant factor, even though it is quite elementary. 
\begin{corollary}\label{zeta bound}
$|\zeta(1/2+it)|\ll \mathfrak{q}(1/2+it)^{1/4}$.
\end{corollary}
\begin{proof}
We will use Theorem~\ref{zeta alg}, but replacing $\mathcal{R}_M(s)$ 
by the correction terms from the 
Euler-Maclaurin formula for $\zeta(s)$
(see the paragraph following the statement
of the theorem). 
We take $s=1/2+it$, $m=0$, $v_0 = u_0 = 4\lceil \sqrt{t}\rceil$,
$M=10 \lceil t\rceil$, and assume that $t\ge 36$, as we may. 
Given our choice of $M$, it is not hard to show that the Euler-Maclaurin 
correction terms contribute $\ll 1$.  
And given our choice of $u_0$, we have $\epsilon_0(s,u_0)\ll 1$.
By routine calculations, $\sum_{n=1}^{v_0-1} n^{-1/2} \le 2\sqrt{v_0}$
and $|\sum_{r=0}^R g_{K_r}(-s/v_r)v_r^{-s}| \le
\mathcal{B}_M(s,u_0,v_0)$.
Thus, 
\begin{equation}
|\zeta(1/2+it)|\ll \sqrt{v_0}+\mathcal{B}_M(1/2+it,u_0,v_0).
\end{equation}
It is helpful to recall that   
$K_r=\lceil v_r/u_0\rceil$ for $r<R$, $v_{r+1}=v_r+K_r$, and
$\mathcal{B}_M(s,u_0,v_0)\le \sum_{r=0}^R
\min\{K_r,|\csc(t/(2v_r))|\}v_r^{-\sigma}$.
So, letting $I_{\ell} := (2^{\ell}u_0, 2^{\ell+1}u_0]$, we 
see that if $v_r\in I_{\ell}$, then $2^{\ell}<K_r \le 2^{\ell+1}$.
We let $I_{\ell_0}$ denote the interval containing $M$, so $M\le
2^{\ell_0+1}u_0$.
Then, using simple estimates, we obtain 
\begin{equation}
\mathcal{B}_M(s,u_0,v_0) 
\le u_0^{-1/2}\sum_{\ell= 0}^{\ell_0} 2^{-\ell/2} 
\sum_{v_r\in I_{\ell}} \min\{2^{\ell+1},|\csc(t/(2v_r)|\}.
\end{equation}
Now, consider that for $v_r,v_{r+1}\in I_{\ell}$, 
with $r<R$, we have 
\begin{equation}
\begin{split}
\frac{t}{2 v_r}-\frac{t}{2 v_{r+1}} = 
\frac{tK_r}{2 v_rv_{r+1}} \ge 
\frac{t2^{\ell}}{2^{2\ell+3} u_0^2} \ge \frac{t}{2^{\ell+3}u_0^2}
\ge \frac{\pi}{B2^{\ell+2}},
\end{split}
\end{equation}
where $B$ is the smallest positive 
integer such that $t/(2\pi u_0^2)\ge  1/B$. Note that, since $u_0=4\lceil \sqrt{t}\rceil$, then $B\ll 1$. 
Also, as $v_r$ ranges over $I_{\ell}$, 
the argument $t/(2\pi v_r)$ moves by increments
$\ge 1/(B2^{\ell+2})$, and it spans 
an interval of length $\le t/(2\pi 2^{\ell} u_0) - t/(2\pi 2^{\ell+1}u_0)
= t/(\pi 2^{\ell+2}u_0)$. Therefore, applying lemma~\ref{csc sum lemma}
to the set $\{v_r\in I_{\ell}\}$ with $Q=B2^{\ell+1}$, $A=B$, and $P=2^{\ell+1}$, we obtain
\begin{equation}
\begin{split}
\sum_{v_r\in I_{\ell}} \min\{2^{\ell+1},|\csc(t/(2(v_r))|\}
\le& (1+ t/(\pi 2^{\ell+2}u_0))(2(B+1)2^{\ell+1}\\
&+ B2^{\ell+2}\log(2^{\ell+1}))\ll t(\ell+1)/u_0.
\end{split}
\end{equation}
It follows that $\mathcal{B}_M(s,u_0,v_0)\le 
(t/(u_0\sqrt{v_0})\sum_{\ell=0}^{\ell_0} (\ell+1)/2^{-\ell/2} \ll
t/(u_0\sqrt{v_0})$.
So we conclude, $\zeta(1/2+it) \ll \sqrt{v_0} + t/(u_0\sqrt{v_0}) \ll t^{1/4}$.
\end{proof}

\section{Parameter choices}
Theorem~\ref{zeta alg} offers a simple method for 
computing $\zeta(\sigma+it)$  with an explicit error bound. 
The control over the error term in the theorem goes beyond what
the Riemann-Siegel asymptotic formula enables.
Theorem~\ref{dirichlet alg}
achieves the same for $L(\sigma+it,\chi)$ when $\chi$ is power-full. 
%This is useful since, as far as we know, there is still no  
%explicit bound for the remainder term in 
%Riemann-Siegel formulas for $L(s,\chi)$. 
%Alternatively, one can use efficient smoothed formula for $L(s,\chi)$, see
%\cite{rubinstein-computational-methods}. 

We implemented a basic version of
Theorem~\ref{zeta alg}  in \verb!Mathematica 9!, which
is an application for computation, see \url{http://www.wolfram.com/mathematica/}. 
This was sufficient for our purposes as we were mainly interested
in learning about reasonable choices of the
parameters. This way, we could appraise the accuracy and running time in practice.
The \verb!Mathematica! notebook containing the implementation 
is available at \url{https://people.math.osu.edu/hiary.1/}.

Our computation relies on finite precision arithmetic, which
introduces round-off errors. Such errors become 
significant for large $t$. This is primarily because 
the computation of $t\log n\bmod{2\pi}$  will contain only a few correct digits
for large $t$.
In general, one cannot expect more than $\pm \epsilon_{mach}\, t\log n$ accuracy
when computing $e^{it\log n}$, where
$\epsilon_{mach}$ is the machine epsilon. 
So if $t>1/\epsilon_{mach}$ say, then, certainly, 
numerical results will not be meaningful. 
To overcome this problem, one could 
switch to an arithmetic system with a smaller machine epsilon 
(but having a slower performance).
Assuming that round-off errors behave like independent random
variables, which is a reasonable model, 
the accumulated round-off error in computing 
$\sum_{n<M} n^{-1/2-it}$ will be typically like
$\pm \epsilon_{mach}\,t(\sum_{n<M} (\log n)^2/n)^{1/2}$.
For double-precision arithmetic, 
$\epsilon_{mach} = 2^{-52} \approx 2 \times 10^{-16}$.
So, if we use double-precision arithmetic with $t=10^d$ and $M\approx 10t$, 
the accumulated round-off error will be like $\pm 10^{d-16} \log
10^{3(d+1)/2}$. 

With this in mind, we obtained marginally better control over the round-off errors 
by using the main sum from the Euler-Maclaurin formula
with $6$ correction terms, and with $M=10\lceil \mathfrak{q}(s)\rceil$,
in particular we did not need to take $M$ very large.
We computed $g_K(z)$ using the formula $(e^{Kz}-1)/(e^z-1)$ when
$|z|>10(m+1)/(K-1)$ (as is typically the case), and 
using the \verb!Mathematica! built-in 
Euler-Maclaurin summation routine when $|z|<10(m+1)/(K-1)$. 
To check the accuracy of the results,
we compared them with the outputs from \verb!lcalc!
and the \verb!Mathematica! built-in zeta routine,
leading to Table~\ref{compos}.
%We considered the practical accuracy of the formula as $m$ and $t$ varied.
We attempted to increase the accuracy 
by inputting $t$ in \verb!Mathematica!  using a higher precision.
However, it is likely that 
\verb!Mathematica! still uses double-precision arithmetic in 
intermediate steps and some built-in routines. 
So the accuracy of many stages of the computation will be limited 
by the machine epsilon for double-precision numbers. 
%This is probably why, when $t$ is large, the accuracy 
%improves less with $m$, since
%the accumulated round-off error will dominate
%the truncation error $\mathcal{T}_{M,m}(s,u_0,v_0)$;
%see Table~\ref{compos}.

The error entries in Table~\ref{compos} are
significantly smaller than the explicit
bound for $\mathcal{T}_{M,m}(s,u_0,v_0)$ given in Theorem~\ref{zeta alg}.
For example, when $t=10^{10}$ and
$m=6$, the explicit bound gives $|\mathcal{T}_{M,m}(s,u_0,v_0)|\le 2.9\times
10^{-3}$ (here, we calculated $\mathcal{B}_M(s,u0,v0)$ directly).
This is significantly larger than the observed error $1.9\times 10^{-10}$ in
Table~\ref{compos}. This is not surprising, and is due to the pseudo-random nature
of round-off errors. 

\begin{table}[ht]
\renewcommand\arraystretch{1.5}
\footnotesize
\caption{Error for various $t$ and $m$, and using $\sigma=1/2$, 
$u_0 = 6\lceil \sqrt{\mathfrak{q}(s)}\rceil$, and $v_0=10(m+1)u_0$.}
\label{compos}
\begin{tabular}{l|llll}
$t$ & $m=0$ & $m=2$ & $m=4$ & $m=6$  \\
\hline
$10^4$    & $3.0\times 10^{-4}$ & $1.7\times 10^{-6}$ & $5.8\times 10^{-9}$ & $3.7\times 10^{-11}$ \\ 
$10^6$    & $1.2\times 10^{-2}$ & $1.6\times 10^{-5}$ & $7.0\times 10^{-9}$ & $6.9\times 10^{-12}$ \\
$10^8$    & $1.9\times 10^{-2}$ & $2.7\times 10^{-5}$ & $2.8\times 10^{-7}$ & $9.4\times 10^{-10}$ \\
$10^{10}$ & $5.4\times 10^{-3}$ & $1.6\times 10^{-5}$ & $4.2\times 10^{-8}$ & $1.9\times 10^{-10}$ \\
\end{tabular}
\end{table}

There was no attempt to optimize our implementation 
since, in any case, it is 
not competitive with an implementation directly in \verb!C/C++!. 
With our parameter choices, and for large $t$, the implementation
was slower by factor of about $2(m+1)^2\log t$ compared to
 computing the main sum in a Riemann-Siegel formula directly
(in both cases we input $t$ in higher precision than double-precision).
The implementation was  
faster by a factor of about $10 \sqrt{t}/((m+1)^2\log t)$ than computing
$\sum_{n\le M} n^{-s}$ directly (this is essentially the 
main sum in the Euler-Maclaurin formula).
It might be possible to speed up the implementation by a factor of $m+1$ if
the derivatives $g_K^{(j)}(z)$, $0\le j\le m$, are computed simultaneously via 
a recursion. One can also save a factor of $2$ by choosing $u_0=3\lceil
\sqrt{\mathfrak{q}(s)}\rceil$ instead of $u_0=6\lceil
\sqrt{\mathfrak{q}(s)}\rceil$, at the expense of a larger truncation
error $\mathcal{T}_{M,m}(s,u_0,v_0)$.

\bibliographystyle{amsplain}
\bibliography{simpleAlg}
\end{document}